\documentclass[11pt,reqno,a4paper]{amsart}%
\usepackage{amsfonts,amsmath,times}
\usepackage{tikz}

\usepackage[sc]{mathpazo} 
\usepackage[scaled]{helvet} 
\usepackage{eulervm} 

\definecolor{grau}{rgb}{0.65,0.65,0.65}
\definecolor{dblau}{rgb}{0,0,0.45}
\definecolor{blau}{rgb}{0,0,0.75} 
\definecolor{grun}{rgb}{0.1,0.6,0.1} 
\usepackage[pdftex]{hyperref}
\hypersetup{colorlinks,linkcolor=grun,citecolor=blue,urlcolor=blau}

\newcommand{\myt}[1]{\emph{\color{dblau}#1}}

\theoremstyle{plain}
\newtheorem{lem}{\normalfont\scshape Lemma}
\newtheorem{thm}{\normalfont\scshape Theorem}
\newtheorem{coroll}{\normalfont\scshape Corollary}
\newtheorem{prop}{Proposition}

\theoremstyle{definition}
\newtheorem{remark}{\normalfont\scshape Remark}
\newtheorem{example}{\normalfont\scshape Example}
\newtheorem{defi}{\normalfont\scshape Definition}

\def\zt{\zeta}
\def\zts{\zeta^{\star}}
\def\Q{\mathbf Q}
\def\R{\mathbf R}

\newcommand{\Ha}[2]{\ensuremath{H_{#1}^{(#2)}}}

\newcommand{\law}{\ensuremath{\stackrel{(d)}=}}

\newcommand{\hb}[3]{\ensuremath{(#1,#2)_{#3}}}

\newcommand{\N}{\ensuremath{\mathbb{N}}}
\newcommand{\E}{\ensuremath{\mathbb{E}}}
\newcommand{\V}{\ensuremath{\mathbb{V}}}

\newcommand{\fallfak}[2]{\ensuremath{#1^{\underline{#2}}}}
\newcommand{\auffak}[2]{\ensuremath{#1^{\overline{#2}}}}
\newcommand{\stir}[2]{\genfrac{ [ }{ ] }{0pt}{}{#1}{#2}}

\DeclareMathOperator{\id}{\text{id}}

\DeclareMathOperator{\CT}{\text{CT}}
\DeclareMathOperator{\Bin}{\mathcal{BIN}}
\DeclareMathOperator{\Co}{\mathcal{C}}

\begin{document}
\title{Logarithmic integrals, zeta values, and tiered
binomial coefficients}%

\author[M.~E.~Hoffman]{Michael E. Hoffman}
\address{Michael E. Hoffman\\
U. S. Naval Academy\\
Annapolis, MD 21402 USA}
\email{meh@usna.edu}

\author[M.~Kuba]{Markus Kuba}
\address{Markus Kuba\\
Department Applied Mathematics and Physics\\
FH - Technikum Wien\\
H\"ochst\"adtplatz 5\\
1200 Wien, Austria} %
\email{kuba@technikum-wien.at}

\date{\today}

\begin{abstract}
We study logarithmic integrals of the form 
$\int_0^1 x^i\ln^n(x)\ln^m(1-x)dx$. They are expressed
as a rational linear combination of certain rational numbers $\hb{n}mi$, which we call tiered binomial
coefficients, and products of the zeta values $\zeta(2)$, $\zeta(3)$,\dots. Various properties of 
the tiered binomial coefficients are established. They involve, amongst others, the binomial transform, truncated multiple zeta and multiple zeta star values, as well as special functions. As an application we discuss the limit law of the number of comparisons of the Quicksort algorithm: we reprove that the moments of the limit law are rational polynomials in the zeta values. A novel expression for the cumulants of the Quicksort limit is also presented. 
\end{abstract}


\maketitle

\section{Introduction}
The \myt{multiple zeta values}~\cite{H92,Z} are defined by
\[
\zt(i_1,\dots,i_k)=\sum_{\ell_1>\cdots>\ell_k\ge 1}\frac1{\ell_1^{i_1}\cdots \ell_k^{i_k}},
\]
with admissible indices $(i_1,\dots,i_k)$ satisfying $i_1\ge 2$, $i_j\ge 1$ for $2\le j\le k$.
Their \myt{truncated} counterparts, often also called multiple harmonic sums, are defined by
\[
\zt_n(i_1,\dots,i_k)=\sum_{n\ge \ell_1>\cdots>\ell_k\ge 1}\frac1{\ell_1^{i_1}\cdots \ell_k^{i_k}}.
\]
We refer to $i_1 +\dots + i_k$ as the weight of this multiple zeta value, and $k$ as its depth. 
For an overview as well as a great many pointers to the literature we refer to the articles~\cite{H2005,H2018,Zu}.

\smallskip

An important variation of the (truncated) multiple zeta values are the so-called (truncated) multiple zeta \myt{star values}. 
Here, equality of the indices is allowed:
\[
\begin{split}
  \zts(i_1,\dots,i_k) & = \sum_{\ell_1\ge \cdots\ge \ell_k\ge 1}\frac1{\ell_1^{i_1}\cdots \ell_k^{i_k}} \quad \text{and}\\
  \zts_{n}(i_1,\dots,i_k) & =\sum_{n \ge \ell_1\ge \cdots\ge \ell_k\ge 1}\frac1{\ell_1^{i_1}\cdots \ell_k^{i_k}}.
\end{split}
\]


In this work we study the logarithmic integrals $I_{n,m}^{(i)}$, defined as
\begin{equation}
\label{Def1}
I_{n,m}^{(i)}  :=\int_0^1 x^i\cdot \ln^n(x)\cdot \ln^m(1-x)\ dx,
\end{equation}
with $i\ge -1$, $n,m\in\N_0$, and their properties. Particular instances of such have been studied previously by K\"olbig~\cite{Koe82,Koe85,K1986} and also by Laurenzi~\cite{L2010}, who discussed special cases of the instance $i=0$ (without obtaining a closed formula). Xu~\cite{Xu2017} studied, amongst others, the case $i=-1$ and related it to multiple zeta values. We note here
that our results also cover the variation $
\int_0^1 x^i\cdot \ln^n(x)\cdot \frac{\ln^m(1-x)}{1-x}\ dx$ of the integral, due to
\[
\int_0^1 x^i\cdot \ln^n(x)\cdot \frac{\ln^m(1-x)}{1-x}\ dx=\frac{1}{m+1}\Big(i\cdot I_{n,m+1}^{(i-1)}+n\cdot I_{n-1,m+1}^{(i-1)}\Big).
\]

\smallskip

\subsection{Main results}
Let $S_{n,m}^{(i)}$ denote the normalized values 
\[
S_{n,m}^{(i)}:=\frac{(-1)^{n+m}}{n!m!}\cdot I_{n,m}^{(i)}.
\]
We summarize our main results: an expansion of the normalized logarithmic integrals $S_{n,m}^{(i)}$ into multiple zeta values.
\begin{thm}
The normalized logarithmic integrals $S_{n,m}^{(i)}$ are given by
\label{ThmMain}
\[
S_{n,m}^{(i)}=\hb{n}mi - \sum_{\substack{1\le a\le n\\a\le b\le m}}
\hb{n-a}{m-b}i\cdot\zt(a+1,\{1\}_b).
\]
Here the values $\hb{n}mi$ denote certain rational numbers, called \myt{binomial coefficients of tier $i$}, 
given by
\[
\hb{n}mi=\frac{1}{i+1}\sum_{k=0}^{n}(-1)^k\binom{n-k+m}{m}\zt_{i}(\{1\}_{k})\zts_{i+1}(\{1\}_{n-k+m}).
\]
\end{thm}
Various explicit expressions in terms of truncated multiple zeta values $\zt_n(\{1\}_k$ and star values $\zts_n(\{1\}_k$
are later on established. As a byproduct of our study we establish the following.
\begin{coroll}
\label{CorollRatioPoly}
The (normalized) logarithmic integrals $S_{n,m}^{(i)}$, $n,m\ge 0$ and $i\ge -1$ are rational polynomials in the ordinary zeta values $\zt(2),\zt(3)$, $\dots$.
\end{coroll}

This will follow directly from earlier results in the literature. Borwein, Bradley and Broadhurst proved that for all positive integers $n,m$ the multiple zeta value $\zt(m+1,\{1\}_n)$ is a rational polynomial in the $\zt(i)$~\cite[Eq. (10)]{BBB}:
\begin{equation}
\label{SUMbor}
\sum_{m,n\ge 0}\zt(m+2,\{1\}_n)x^{m+1}y^{n+1}=1-\exp\biggl(\sum_{k\ge 2}\frac{x^k+y^k-(x+y)^k}{k}\zeta(k)\biggr).
\end{equation}
We mention a recently obtained explicit expression by Kaneko and Sakata~\cite{KaSa}:
\[
\zt(m+1,\{1\}_{n-1})=\sum_{i=1}^{\min(m,n)}(-1)^{i-1}
\sum_{\substack{\text{wt}(\mathbf{m})=m,\text{wt}(\mathbf{n})=n 
\\\text{dep}(\mathbf{m})=\text{dep}(\mathbf{n})=i }}
\zt(\mathbf{m}+\mathbf{n}),
\]
where for two indices $\mathbf{m} = (m_1,\dots,m_i)$ and $\mathbf{n} =(n_1,\dots, n_i)$ with weights $\text{wt}(\mathbf{m}) = \sum_{j} m_{j} = m$ and $\text{wt}(\mathbf{n}) = \sum_{j} n_{j} = n$, respectively, which have the same depth $\text{dep}(\mathbf{m})=\text{dep}(\mathbf{n})=i$, the sum $\mathbf{m} + \mathbf{n}$ denotes $(m_1 + n_1,\dots, m_i + n_i)$.

\smallskip 

After our main results, we also turn to extensions. We generalize some results for $S_{n,m}^{(i)}$ to
a generalized version of the Nielsen polylogarithm. Moreover, we analyze the logarithmic integrals
with negative powers of $i$ and obtain the following extensions of Corollary~\ref{CorollRatioPoly}.
s\begin{coroll}
\label{CorollNegativePowers}
The (normalized) logarithmic integrals $S_{n,m}^{(i)}$, with $n,m\ge 0$ and $i\ge -m$, are rational polynomials in the ordinary zeta values $\zt(2),\zt(3)$, $\dots$.
\end{coroll}

\smallskip 

\subsection{Structure and Notation}
This article is structured as follows. First, we turn to the integrals $I_{n,m}^{(i)}$ and $S_{n,m}^{(i)}$ and derive various
properties of them. We use special Hurwitz zeta values to obtain a recurrence relation for $S_{n,m}^{(i)}$.
This recurrence relation is then translated into a recurrence for the the binomials coefficients of tier $i$. 
The enumeration of $\hb{n}mi$ is then solved using generating functions. Various additional properties of these numbers are then given. 
They involve, amongst others, the binomial transform, truncated multiple zeta and multiple zeta star values, as well as Euler polynomials and Legendre polynomials. 
Then, we discuss extensions of our results: we discuss negative values of $i$, as well as an extension to a generalization of the Nielsen polylogarithm function.
In the final section we discuss applications of our results to the limit law for the number of comparisons in the Quicksort algorithm.

\smallskip 

A basic ingredient of our computations are the Stirling numbers $\stir{n}{k}$ of the first kind, also called Stirling cycle numbers.
They count the number of permutations of $n$ elements with $k$ cycles~\cite{GKP} and appear as coefficients in the expansions
\begin{equation*}
\fallfak{x}n=\sum_{k=0}^{n}(-1)^{n-k}\stir{n}k x^k=\sum_{k=0}^{n}s(n,k) x^k,
\end{equation*}
relating ordinary powers $x^n$ to the so-called falling factorials $\fallfak{x}{n}=x(x-1)\dots(x-(n-1))$, for integers $n \ge 1$, and $\fallfak{x}{0}=1$. The definition can be extended to negative integers via $\fallfak{x}{-n} = \frac{1}{\fallfak{(x+n)}{n}}$, $n \ge 1$.
Here $s(n,k)$ denote the signed Stirling numbers.

\section{Evaluations of the logarithmic integrals}
We start with special instances of the integral, which are well known. 
\begin{lem}[Boundary values - Case $m=0$]
For $m=0$ and $i\ge 0$ we have 
\begin{equation}
\label{SpecialValue}
I_{n,0}^{(i)}  :=\int_0^1 x^i\cdot \ln^n(x)dx=\frac{(-1)^n n!}{(i+1)^{n+1}}. 
\end{equation}
\end{lem}
The standard proof of this folklore result uses repeated integrated by parts and induction. Alternatively, this can be obtained by using the Beta integral $B(\alpha,\beta)$, with
\[
B(u+1,v+1):=\int_0^{1}x^{u}(1-x)^{v}dx=\frac{\Gamma(u+1)\Gamma(v+1)}{\Gamma(u+v+2)}.
\]
We note that the following folklore argument already appears in~\cite{He91}, and has been rediscovered and used several times~\cite{F2015,L2010}.
\begin{proof}
Let $E_{u=a,v=b}=E_{u=a}E_{v=b}$ denote the evaluation operator at $u=a$ and $v=b$.
Then, 
\[
I_{n,m}^{(i)}=E_{u=i,v=0}\frac{\partial^{n+m}}{\partial u^n\partial v^m}B(u+1,v+1)
\]
Thus, for $m=0$ we get
\begin{align*}
I_{n,0}^{(i)}& =E_{u=i,v=0}\frac{\partial^{n}}{\partial u^n}B(u+1,v+1)
=E_{u=i}\frac{\partial^{n}}{\partial u^n}\frac{\Gamma(u+1)}{\Gamma(u+2)}\\
&=E_{u=i}\frac{\partial^{n}}{\partial u^n}\frac{1}{u+1}
=\frac{(-1)^n n!}{(i+1)^{n+1}}.
\end{align*}
\end{proof}

\subsection{Symmetry of the integral and boundary values}
The symmetry of the logarithmic integrals $I_{n,m}^{(i)}$ involves the binomial transform, as introduced by Knuth~\cite{Kn73}.
\begin{defi}[Binomial transform]
The binomial transform $\Bin$ of a sequence $(a_i)_{i\ge 0}$ is a sequence $(s_i)_{i\ge 0}$, defined by 
\[
s_i=\Bin(a_i)=\sum_{k=0}^{i}(-1)^k\binom{i}k a_k.
\] 
\end{defi}
Note that the binomial transform is evidently a linear operator and an involution: $\Bin\circ\Bin=\id$. We refer the reader to
the article of the first author~\cite{H2005} for algebraic properties and to Prodinger~\cite{Pro1994} for additional properties concerning generating functions. 

\smallskip

Additionally, given a double sequence $(a_{n,m})_{n,m\ge 0}$, let $\Co$ denote the operator 
which exchanges the indices:
\[
\Co(a_{n,m})=a_{m,n},
\]
such that $\Co\circ\Co=\id$.
\begin{prop}[Generalized Symmetry of logarithmic integrals]
\label{Prop1Symmetry}
The logarithmic integrals $I_{n,m}^{(i)}$ satisfy the generalized symmetry relation
\[
I_{n,m}^{(i)}= \sum_{j=0}^i \binom{i}{j}(-1)^j \cdot I_{m,n}^{(j)}.
\]
Equivalently, concerning the logarithmic integrals, the binomial transform $\Bin$ with respect to the variable $i$, equals the operator $\Co$ with respect to $n$ and $m$:
\[
\Co ( I_{n,m}^{(i)})= \Bin (I_{n,m}^{(i)}).
\]
\end{prop}
\begin{proof}
We use the substitution $x=1-u$ and readily obtain
\[
I_{n,m}^{(i)}  =\int_0^1 x^i\cdot \ln^n(x)\cdot \ln^m(1-x)\ dx
=\int_0^1 (1-u)^i\cdot\ln^m(u)\cdot \ln^n(1-u)du.
\]
Expansion of the term $(1-u)^i=\sum_{j=0}^i \binom{i}{j}(-1)^j u^j$ gives the stated result.
\end{proof}
A direct byproduct of the generalized symmetry is the evaluation of $I_{0,m}^{(i)}$. 
\begin{thm}[Boundary values - case $n=0$]
\label{TheNull}
The values $I_{0,m}^{(i)}$ are given by truncated zeta star series,
\[
I_{0,m}^{(i)}= (-1)^m m!\cdot\frac1{i+1}\cdot\zts_{i+1}(\{1\}_m).
\]
\end{thm}
Before we turn to the proof we collect a well known result~\cite{D,FS,KP}.
\begin{lem}[Truncated multiple zeta values $\zts_m(\{1\}_k)$]
\label{LemZTS}
For positive integers $n \ge k$, the values $\zts_{n}(\{1\}_{k})$ can be expressed as 
\[
\zts_{n}(\{1\}_{k})=\sum_{j=1}^{n}\binom{n}j\frac{(-1)^{j-1}}{j^k}.
\]
\end{lem}
\begin{proof}
By the symmetry relation, 
\[
I_{0,m}^{(i)}= \sum_{j=0}^i \binom{i}{j}(-1)^j \cdot I_{m,0}^{(j)}.
\]
The latter integrals are well known~\eqref{SpecialValue}
and we get
\[
I_{0,m}^{(i)}= \sum_{j=0}^i \binom{i}{j}(-1)^j \cdot \frac{(-1)^m m!}{(j+1)^{m+1}}.
\]
Since $\binom{i}{j}=\frac{j+1}{i+1}\binom{i+1}{j+1}$ we get
\begin{align*}
I_{0,m}^{(i)}&= (-1)^m m!\cdot\frac{1}{i+1}\cdot\sum_{j=0}^i \binom{i+1}{j+1}(-1)^j \cdot \frac{1}{(j+1)^{m}}\\
&= (-1)^m m!\cdot\frac{1}{i+1}\cdot\sum_{j=1}^{i+1} \binom{i+1}{j}(-1)^{j+1} \cdot \frac{1}{j^{m}}.
\end{align*}
Using the Lemma stated before we obtain the stated result.
\end{proof}

\subsection{A representation using Stirling numbers}
Our aim is to prove the following presentation.
\begin{prop}[Logarithmic integrals and truncated zeta values]
\label{PropStir}
The logarithmic integrals $I_{n,m}^{(i)}$ satisfy
\[
I_{n,m}^{(i)} = n!m!(-1)^{m+n}\sum_{\ell=1}^{\infty}\frac{\zt_{\ell-1}(\{1\}_{m-1})}{\ell(\ell+i+1)^{n+1}}.
\]
In particular, for $i=-1$ we have
\[
I_{n,m}^{(-1)}=n!m!(-1)^{m+n}\zt(n+2,\{1\}_{m-1}).
\]

\smallskip 

We have the alternative expression
\[
I_{n,m}^{(i)} = n!m(-1)^{n-1}\sum_{\ell=1}^{\infty}\frac{B_{m-1}(-0!\Ha{\ell-1}{1},-1!\Ha{\ell-1}{2},\dots,-(m-2)!\Ha{\ell-1}{m-1})  }{\ell(\ell+i+1)^{n+1}},
\]
where the $B_k$ denote the Bell polynomials. 
\end{prop}
In order to evaluate the logarithmic integrals we proceed similar to~\cite{H2018+}.
We can interpret the power $\left(\log(1-x)\right)^k$ (see~\cite[p. 351]{GKP}) 
as the row generating function of the Stirling numbers of the first kind.
\begin{equation}
\label{logp}
\left(\log(1-x)\right)^k=k!\sum_{\ell=1}^\infty \frac{(-1)^k x^\ell}{\ell!}\stir{\ell}k.
\end{equation}
We use a well-known relation between Stirling numbers of the first kind and truncated multiple zeta values;
see for example~\cite{H2018+}, or~\cite{KP} for more properties of $\stir{n}k$ and simple proofs.
\begin{lem}[Truncated multiple zeta values $\zt_n(\{1\}_k)$]
\label{Lem1Stir}
For positive integers $n\ge k$ the Stirling numbers of the first kind are related to truncated multiple zeta values:
\[
\stir{n}{k}=(n-1)!\zt_{n-1}(\{1\}_{k-1}),
\]
where $\{1\}_m$ means 1 repeated $m$ times. Moreover, $\zt_{n}(\{1\}_{k})$
is given in terms of Bell polynomials and generalized Harmonic numbers, 
\begin{align*}
\zt_{n}(\{1\}_{k})&=\frac{(-1)^{k}}{k!}B_{k}(-0!\Ha{n}{1},-1!\Ha{n}{2},\dots,-(k-1)!\Ha{n}{k})\\
&=\sum_{m_1+2m_2+\dots=k}\frac{(-1)^{m_2+m_4+\dots}}{m_1!m_2!\dots}\Big(\frac{\Ha{n}{1}}{1}\Big)^{m_1} \Big(\frac{\Ha{n}{2}}{2}\Big)^{m_2}\dots.
\end{align*}
\end{lem}
Now we can rewrite the logarithmic integral as follows:
\[
I_{n,m}^{(i)}  =\int_0^1 x^i\cdot \ln^n(x)\cdot \ln^m(1-x)\ dx
=m!(-1)^m\sum_{\ell=1}^{\infty}\frac{\stir{\ell}m}{\ell!} \int_0^1 x^{\ell+i}\ln^n(x).
\]
We use the known special values~\eqref{SpecialValue} 
to get the desired representation
\[
I_{n,m}^{(i)} =m!(-1)^m\sum_{\ell=1}^{\infty}\frac{\stir{\ell}m}{\ell!}\cdot \frac{(-1)^n n!}{(\ell+i+1)^{n+1}}
= n!m!(-1)^{m+n}\sum_{\ell=1}^{\infty}\frac{\zt_{\ell-1}(\{1\}_{m-1}}{\ell(\ell+i+1)^{n+1}}.
\]
The special case $i=-1$ follows easily from the definition of the multiple zeta values.

\subsection{Partial fraction decomposition and recurrence relations}
In the following we further study the normalized values $S_{n,m}^{(i)}$, given by
\[
S_{n,m}^{(i)} = \sum_{\ell=1}^{\infty}\frac{\zt_{\ell-1}(\{1\}_{m-1})}{\ell(\ell+i+1)^{n+1}}.
\]

\smallskip

We introduce a variant $T_{n,m}^{(i)}$ of the values $S_{n,m}^{(i)}$, defined by 
\[
T_{n,m}^{(i)} = \sum_{\ell=1}^{\infty}\frac{\zt_{\ell-1}(\{1\}_{m-1})}{(\ell+i+1)^{n+1}},
\]
for $n,m\ge 1$ and $i\ge -1$. Its initial values are given by $T_{n,m}^{(-1)}=S_{n-1,m}^{(-1)}=\zt(n+1,\{1\}_{m-1})$.

\begin{prop}[Relations between $S$ and $T$ values]
\label{LemRelationST}
The normalized logarithmic integrals $S_{n,m}^{(i)}$ and the values $T_{n,m}^{(i)}$ are related by
\[
S_{n,m}^{(i)}=T_{n,m+1}^{(i-1)}-T_{n,m+1}^{(i)},
\]
for $m,n,i\ge 0$. Consequently, for $m\ge 1$ we have
\[
T_{n,m}^{(i)}=-\sum_{j=0}^{i-1}S_{n,m-1}^{(j)}+\zt(n+1,\{1\}_{m-1}).
\]
On the other hand, for $i\ge 0$, $n\ge 1$ there is also the relation
\[
S_{n,m}^{(i)}=\frac1{i+1}\cdot S_{n-1,m}^{(i)}-\frac1{i+1}\cdot T_{n,m}^{(i)}.
\]
\end{prop}
\begin{remark}
The multiple Hurwitz zeta values are defined by
\[
\zt_H(i_1,\dots,i_k;p_1,\dots,p_k)=
\sum_{\ell_1>\cdots>\ell_k\ge 1}\frac1{(\ell_1+p_1)^{i_1}\cdots (\ell_k+p_k)^{i_k}}.
\]
Evidently, they are related to the ordinary multiple zeta values by 
\[
\zt_H(i_1,\dots,i_k;0,\dots,0)=\zt(i_1,\dots,i_k).
\]
The $T$-values introduced before are special instances of the multiple Hurwitz zeta values:
\begin{equation}
\label{Hurwitz}
T_{n,m}^{(i)} = \sum_{\ell=1}^{\infty}\frac{\zt_{\ell-1}(\{1\}_{m-1})}{(\ell+i+1)^{n+1}}
=\zt_H(n+1,\{1\}_{m-1},i+1;\{0\}_{m-1}).
\end{equation}
\end{remark}

\begin{proof}
The truncated zeta values $\zt_{\ell}(\{1\}_{m})$ satisfy the recurrence relation
\[
\zt_{\ell}(\{1\}_{m})=\sum_{k=1}^{\ell}\frac{\zt_{k-1}(\{1\}_{m-1})}k =  \zt_{\ell-1}(\{1\}_{m})+\frac{1}\ell\zt_{\ell-1}(\{1\}_{m-1}).
\]
Hence, 
\[
\frac{1}\ell\zt_{\ell-1}(\{1\}_{m-1})=\zt_{\ell}(\{1\}_{m})-\zt_{\ell-1}(\{1\}_{m}),
\]
This implies that 
\[
S_{n,m}^{(i)} = \sum_{\ell=1}^{\infty}\frac{\zt_{\ell}(\{1\}_{m})}{(\ell+i+1)^{n+1}}-T_{n,m+1}^{(i)}
=T_{n,m+1}^{(i-1)}-T_{n,m+1}^{(i)}.
\]
Let $\nabla$ denote the backward difference operator with respect to the variable $i$, 
such that $-\nabla T_{n,m}^{(i)}=S_{n,m-1}^{(i)}$. The inverse operator is the summation operator~\cite{GKP}, leading directly 
to the stated result for $T_{n,m}^{(i)}$.

\smallskip

For the second recurrence relation we use
the partial fraction decomposition
\[
\frac{1}{\ell(\ell+i+1)}
=\frac{1}{i+1}\cdot\bigg[\frac1\ell-\frac1{\ell+i+1}\bigg],
\]
leading directly to the stated result
\begin{align*}
S_{n,m}^{(i)}&=
\frac1{i+1}\bigg[S_{n-1,m}^{(i)}-T_{n,m}^{(i)}\bigg].
\end{align*}
\end{proof}

Next we obtain a recurrence relation for the $S$-values. 
\begin{prop}[Recurrence relation for $S$-values]
\label{PropHRecurrence}
The normalized logarithmic integrals $S_{n,m}^{(i)}$ satisfy for $n,m>0$ and $i\ge 0$ the recurrence relation
\[
S_{n,m}^{(i)}=\frac{1}{i+1}\cdot
\bigg[
S_{n-1,m}^{(i)} +\sum_{j=0}^{i}S_{n,m-1}^{(j)}-\zt(n+1,\{1\}_{m-1})
\bigg],
\]
with initial values $S_{0,m}^{(i)}= \frac1{i+1}\cdot\zts_{i+1}(\{1\}_m)$ and $S_{n,0}^{(i)}  =\frac{1}{(i+1)^{n+1}}$.
\end{prop}

\begin{proof}
We simply plug the expression for $T_{n,m}^{(i)}$,
\[
T_{n,m}^{(i)}=-\sum_{j=0}^{i-1}S_{n,m-1}^{(j)}+\zt(n+1,\{1\}_{m-1}).
\]
into
\[
S_{n,m}^{(i)}=\frac1{i+1}\cdot S_{n-1,m}^{(i)}-\frac1{i+1}\cdot T_{n,m}^{(i)}
\]
and get the stated recurrence relation. The initial values are known and given in\eqref{SpecialValue} and Theorem~\ref{TheNull}.
\end{proof}

\subsection{A nested sum expression}
A direct byproduct of our recurrence relation and the initial value for $n=0$ is a reformulation of the recurrence relation to a ``one step'' recurrence relation with respect to $m$.
\begin{prop}
The normalized logarithmic integrals $S_{n,m}^{(i)}$ satisfy for $n,m>0$ and $i\ge 0$ the recurrence relation
\[
S_{n,m}^{(i)}=
\sum_{r=1}^{n}\frac{1}{(i+1)^{n+1-r}}
\bigg[\sum_{j=0}^{i}S_{r,m-1}^{(j)}-\zt(r+1,\{1\}_{m-1})\bigg]
+\frac1{(i+1)^{n+1}}\cdot\zts_{i+1}(\{1\}_m).
\]
\end{prop}
\begin{proof}
We interpret the recurrence relation in Proposition~\ref{PropHRecurrence} 
as a recurrence relation for $a_n=S_{n,m}^{(i)}$ with respect to $n$:
\[
a_n = \frac{1}{i+1}\cdot a_{n-1}+t_n, 
\]
with toll function $t_n$ given by
\[
\frac1{i+1}\bigg[\sum_{j=0}^{i}S_{n,m-1}^{(j)}-\zt(n+1,\{1\}_{m-1})\bigg].
\]
The well known solution of the recurrence relation
is given by 
\[
a_n=\sum_{r=1}^{n}\frac{t_{r}}{(i+1)^{n-r}} + \frac1{(i+1)^{n}}a_0,
\]
leading to our result. 
\end{proof}
We modify the recurrence relation further: 
\[
S_{n,m}^{(i)}=
\sum_{r=1}^{n}\frac{1}{(i+1)^{n+1-r}}\sum_{j=0}^{i}S_{r,m-1}^{(j)}
+t_{n,m,i},
\]
with toll function $t_{n,m,i}$ given by
\begin{equation}
\label{DefToll}
t_{n,m,i}=-\sum_{r=1}^{n}\frac{1}{(i+1)^{n+1-r}}\zt(r+1,\{1\}_{m-1})+\frac1{(i+1)^{n+1}}\cdot\zts_{i+1}(\{1\}_m).
\end{equation}
For small values of $m$ we can use $S_{n,0}^{(i)}  =\frac{1}{(i+1)^{n+1}}$ to obtain explicit expressions.
\begin{example}[Case $m=1$].
The normalized logarithmic integrals $S_{n,1}^{(i)}$ are given by
\begin{align*}
S_{n,1}^{(i)}&=
\sum_{r=1}^{n}\frac{1}{(i+1)^{n+1-r}}\sum_{j=0}^{i}\frac{1}{(j+1)^{r+1}}\\
&\quad-\sum_{r=1}^{n}\frac{1}{(i+1)^{n+1-r}}\zt(r+1)+\frac1{(i+1)^{n+1}}\cdot\zts_{i+1}(1).
\end{align*}
\end{example}
\begin{example}[Case $m=2$].
The normalized logarithmic integrals $S_{n,1}^{(2)}$ are given by
\begin{align*}
S_{n,2}^{(i)}&=
t_{n,2,i} + \sum_{r=1}^{n}\frac{1}{(i+1)^{n+1-r}}\sum_{j=0}^{i}t_{r,1,j}\\
&\quad+\sum_{r_1=1}^{n}\sum_{j_1=0}^{i}\sum_{r_2=1}^{r_1}\sum_{j_2=0}^{j_1}\frac{1}{(i+1)^{n+1-r_1}(j_1+1)^{r_1+1-r_2}(j_2+1)^{r_2+1}}.
\end{align*}
\end{example}
The general expressions is obtained by iteration with respect to $m$. 
This leads to a nested sum expression.
\begin{thm}[Nested sum expression for logarithmic integrals]
The normalized logarithmic integral $S_{n,m}^{(i)}$ is given by the following nested sums:
\begin{align*}
S_{n,m}^{(i)} &=
\sum_{h=1}^{m-1}
\bigg[\sum_{\substack{n\ge r_1\ge r_{2}\ge \dots \ge r_h \ge 1 \\ i \ge j_1 \ge j_2 \ge \dots \ge j_h\ge 0}}
\frac{t_{r_h,m-h,j_h}}{(j_0+1)^{r_0+1-r_1}(j_1+1)^{r_1+1-r_2}\dots (j_{h-1}+1)^{r_{h-1}+1-r_h}}
\bigg] \\
&+\sum_{\substack{n\ge r_1\ge r_{2}\ge \dots \ge r_m \ge 1 \\ i \ge j_1 \ge j_2 \ge \dots \ge j_m\ge 0}}
\frac{1}{(j_0+1)^{r_0+1-r_1}(j_1+1)^{r_1+1-r_2}\dots (j_{m-1}+1)^{r_{m-1}+1-r_m}(j_m+1)^{r_m+1}}\\
&+t_{n,m,i}.
\end{align*}
Here $j_0:=i$ and $r_0:=n$ with $t_{n,m,i}$ as given in~\eqref{DefToll}.
\end{thm}

\section{Tiered binomial coefficients and small values of i}
The nested sum expressions derived before indicate a different, more compact representation of the normalized logarithmic integrals. 
We define the $i$th tier binomial coefficients as the coefficients of the expansion
\[
S_{n,m}^{(i)}=\hb{n}mi - \sum_{\substack{1\le a\le n\\1\le b\le m}}\hb{n-a}{m-b}{i} \zt(a+1,\{1\}_{b-1}).
\]
According to the initial conditions $S_{0,m}^{(i)}= \frac1{i+1}\cdot\zts_{i+1}(\{1\}_m)$ and $S_{n,0}^{(i)}  =\frac{1}{(i+1)^{n+1}}$
it holds
\[
\hb{0}{m}i=S_{0,m}^{(i)}= \frac1{i+1}\cdot\zts_{i+1}(\{1\}_m)
\]
and 
\[
\hb{n}{0}i=S_{n,0}^{(i)}  =\frac{1}{(i+1)^{n+1}}.
\]

\begin{prop}[Recurrence relation for binomial coefficients of tier i]
\label{PropRec}
The coefficients $\hb{n}mi$ satisfy the recurrence relation
\[
\hb{n}mi = \frac{\hb{n-1}mi}{i+1} + \frac1{i+1}\sum_{j=0}^{i}\hb{n}{m-1}j,\quad n,m>0,\ i\ge 0,
\]
with initial values $\hb{0}{m}i=\frac1{i+1}\cdot\zts_{i+1}(\{1\}_m)$
and $\hb{n}{0}i=\frac{1}{(i+1)^{n+1}}$.
\end{prop}
\begin{remark}
Note that the recurrence relation is actually valid for $n,m\ge 0$, excluding $n=m=0$.
It is sufficient to define $\hb{-1}mi=0$ and $\hb{n}{-1}i=0$.
\end{remark}

\begin{proof}
By the recurrence relation~\eqref{PropHRecurrence}
we obtain
\begin{align*}
&\hb{n}mi - \sum_{\substack{1\le a\le n\\1\le b\le m}}\hb{n-a}{m-b}{i} \zt(a+1,\{1\}_{b-1})\\
&\quad= \frac{\hb{n-1}mi}{i+1} - \sum_{\substack{1\le a\le n-1\\1\le b\le m}}\frac{\hb{n-1-a}{m-b}{i}}{i+1} \zt(a+1,\{1\}_{b-1})\\
&\quad+\frac{1}{i+1}\sum_{j=0}^i\bigg[\hb{n}{m-1}j - \sum_{\substack{1\le a\le n\\1\le b\le m-1}}\hb{n-a}{m-1-b}{j} \zt(a+1,\{1\}_{b-1})\bigg]\\
&-\frac1{i+1}\zt(n+1,\{1\}_{m-1}).
\end{align*}
First we check the boundary cases. For $a=n$ and $b=m$ we
get \[-\hb00i=-\frac{1}{i+1}.\] For $a=n$ and $b<m$ we get 
\[
-\hb{0}{m-b}{i}=-\frac1{i+1}\sum_{j=0}^{i}\hb{0}{m-1-b}{j}.
\]
Due to
\[
\frac1{i+1}\cdot\zts_{i+1}(\{1\}_{m-b})=\frac{1}{i+1}\sum_{j=0}^{i} \frac1{j+1}\cdot\zts_{j+1}(\{1\}_{m-1-b}),
\]
this is compatible with the initial conditions.
For $a<n$ and $b=m$ we have
\[
-\hb{n-a}{0}i=-\frac{\hb{n-1-a}{0}{i}}{i+1},
\]
which is true. For $a=0$ and $b=0$ we directly get the stated recurrence relation.
For the general case $a<n$ and $b<m$ we get 
\[
-\hb{n-a}{m-b}{i}=-\frac{\hb{n-1-a}{m-b}{i}}{i+1} -\frac{1}{i+1}\sum_{j=0}^{i}\hb{n-a}{m-1-b}{j},
\]
which is simply the shifted recurrence relation.
\end{proof}

\subsection{A generating functions approach}
In the following we derive the generating functions of the $i$th tier binomial coefficients.
We begin with two instructive examples and then turn to the general case. For the reader's convenience we also collect 
the corresponding \href{http://oeis.org}{OEIS} links.
\begin{example}[Zero tier - binomial coefficients; \href{http://oeis.org/A007318}{A007318}]
For $i=0$ the coefficients $\hb{n}m0$ satisfy the recurrence relation
\[
\hb{n}m0 = \hb{n-1}m0 + \hb{n}{m-1}0
\]
with initial conditions $\hb{0}{m}0=\hb{n}{0}0=1$. 
One readily obtains, either by guess and prove or by generating functions, the solution 
\[
\hb{n}m0=\binom{n+m}{n}.
\]
We present the simple generating functions proof, which can be regarded as a toy example for small $i$. 
We introduce the generating function $f(x,y)=f_0(x,y)$,
\[
f(x,y)=\sum_{n,m\ge 0}\hb{n} m 0 x^n y^m.
\]
Multiplication of the recurrence relation with $x^ny^m$ and summation over $n,m\ge 0$, excluding $n=m=0$, gives
\[
f(x,y)-1=x\cdot f(x,y) + y\cdot f(x,y).
\]
Consequently, 
\[
f(x,y)=\frac{1}{1-x-y}=\frac{1}{(1-x)(1-\frac{y}{1-x})},
\]
such that 
\[
[x^n y^m]f(x,y)=[x^n]\frac{1}{(1-x)^{m+1}}=\binom{n+m}{m}.
\]
\end{example}
\begin{example}[First tier; \href{http://oeis.org/A308737}{A308737}]
\label{TierOne}
For $i=1$ the coefficients $\hb{n}m1$ satisfy the recurrence relation
\[
\hb{n}m1 = \frac12\hb{n-1}m1 + \frac12\hb{n}{m-1}1+\frac12\hb{n}{m-1}0.
\]
We introduce the generating function $f_1(x,y)$,
\[
f_1(x,y)=\sum_{n,m\ge 0}\hb{n} m 1 x^n y^m.
\]
Multiplication of the recurrence relation with $x^n y^m$ and summation over $n,m\ge 0$, excluding $n=m=0$, gives
\[
f_1(x,y)-\frac12=\frac{x}2\cdot f_1(x,y) + \frac{y}2\cdot f_1(x,y) +\frac{y}2f_0(x,y).
\]
Consequently, 
\begin{align*}
f_1(x,y)&=\frac{\frac12+\frac{y}2f_0(x,y)}{1-\frac{x}2-\frac{y}2}=
\frac{1}{2(1-\frac{x}2-\frac{y}2)}+\frac{y}{2(1-x-y)\big(1-\frac{x}2-\frac{y}2\big)}\\
&=\frac{1-x}{(1-x-y)(2-x-y)}.
\end{align*}
Taylor expansion around $x=y=0$ gives
\begin{align*}
f_1(x,y)&=\frac{1}{2}+\frac{3\cdot y+x}{4}+\frac{7\cdot {{y}^{2}}+8\cdot y\cdot x+{{x}^{2}}}{8}+\frac{15\cdot {{y}^{3}}+31\cdot {{y}^{2}}\cdot x+17\cdot y\cdot {{x}^{2}}+{{x}^{3}}}{16}\\
&\quad+\frac{31\cdot {{y}^{4}}+94\cdot {{y}^{3}}\cdot x+96\cdot {{y}^{2}}\cdot {{x}^{2}}+34\cdot y\cdot {{x}^{3}}+{{x}^{4}}}{32}+\dots.
\end{align*}
In order to extract coefficients we rewrite the generating function
\[
f_1(x,y)=1+\frac{y}{1-x-y}-\frac{1-x}{2(1-\frac{x}2-\frac{y}2)}.
\]
Consequently, we obtain
\[
\hb{n}m1=\binom{n+m-1}{m-1}-\frac{1}{2^{n+m}}\Big[\frac12\binom{n+m}{n}-\binom{n+m-1}{n-1}\Big].
\]
\end{example}

\subsection{General case - explicit expressions}
Introducing the generating function $f_i(x,y)$, defined by
\[
f_i(x,y)=\sum_{n,m\ge 0}\hb{n} m i x^n y^m.
\]
We multiply the recurrence relation~\eqref{PropRec} with $x^n y^m$ and sum over $n,m\ge 0$, excluding $n=m=0$.
This gives a full history recurrence relation for the generating functions $f_i(x,y)$:
\begin{equation}
\label{RecF}
f_i(x,y)-\frac{1}{i+1} = \frac{x}{i+1}\cdot f_i(x,y) + \frac{y}{i+1}\cdot f_i(x,y) + \sum_{j=0}^{i-1}\frac{y}{i+1}f_j(x,y).
\end{equation}
This implies that 
\[
(i+1)\Big(1-\frac{x}{i+1}-\frac{y}{i+1}\Big)f_i(x,y)=1+y\sum_{j=0}^{i-1}f_j(x,y).
\]
Consequently, taking differences lead to an ordinary recurrence relation
\[
\Big(i+1-x-y\Big)f_i(x,y)-\Big(i-x-y\Big)f_{i-1}(x,y)=y f_{i-1}(x,y).
\]
Furthermore, 
\begin{equation}
\label{RecF2}
f_i(x,y)=\frac{i-x}{\big(i+1-x-y\big)}f_{i-1}(x,y),\quad i\ge 1.
\end{equation}
Substituting the result for $f_0(x,y)$ of our first example leads to the following theorem. 

\begin{thm}
The generating function $f_i(x,y)=\sum_{n,m\ge 0}\hb{n} m i x^n y^m$ 
of the $i$th tier binomial coefficients is given by 
\[
f_i(x,y)= \frac{1}{i+1-x-y}\cdot\frac{\binom{i-x}{i}}{\binom{i-x-y}{i}},\quad i\ge 0.
\]
\end{thm}

Alternative expressions for the generating function are
\[
f_i(x,y)
=\frac{1}{1-x-y}\cdot\frac{\binom{x-1}{i}}{\binom{x+y-2}{i}}
=\frac{1}{1-x-y}\cdot\frac{\auffak{(1-x)}i}{\auffak{(2-x-y)}{i}}
=\frac{\fallfak{(i-x)}i}{\fallfak{(i+1-x-y)}{i+1}}.
\]

\begin{example}[Second tier binomial coefficients]
The theorem above gives
\[
f_2(x,y)=\frac{(2-x)(1-x)}{(3-x-y)(2-x-y)(1-x-y)}.
\]
Taylor expansion around $x=y=0$ gives
\begin{align*}
f_2(x,y)&= \frac{1}{3}+\frac{11\cdot y+2\cdot x}{18}+\frac{85\cdot {{y}^{2}}+71\cdot y\cdot x+4\cdot {{x}^{2}}}{108}\\
&\quad+\frac{575\cdot {{y}^{3}}+960\cdot {{y}^{2}}\cdot x+393\cdot y\cdot {{x}^{2}}+8\cdot {{x}^{3}}}{648}\\
&\quad+\frac{3661\cdot {{y}^{4}}+9469\cdot {{y}^{3}}\cdot x+7971\cdot {{y}^{2}}\cdot {{x}^{2}}+2179\cdot y\cdot {{x}^{3}}+16\cdot {{x}^{4}}}{3888}+\mbox{...}
\end{align*}
\end{example}

\begin{prop}
\label{Prop1}
The binomial coefficients $\hb{n}mi$ of tier $i$ are given by formulas
\[
\hb{n}mi=\frac{1}{i+1}\sum_{k=0}^{n}(-1)^k\binom{n-k+m}{m}\zt_{i}(\{1\}_{k})\zts_{i+1}(\{1\}_{n-k+m}),
\]
as well as
\[
\hb{n}mi=i!\sum_{\ell=1}^{i+1}\sum_{k=0}^{n}\frac{(-1)^{\ell+k-1}\zt_{i}(\{1\}_{k})}{(\ell-1)!(i+1-\ell)!}\cdot\frac1{\ell^{n-k+m+1}}\binom{n-k+m}{m}.
\]
\end{prop}
\begin{example}[Case $m=n$, central tiered binomial coefficients]
\label{ExampleCentralTiered}
In the special case of $m=n$ the binomial coefficients $(n,n)_i$ of tier $i$ are given by 
\[
\hb{n}ni=\frac{1}{i+1}\sum_{k=0}^{n}(-1)^k\binom{2n-k}{n}\zt_{i}(\{1\}_{k})\zts_{i+1}(\{1\}_{2n-k}).
\]
\end{example}

In order to obtain a closed form expressions for $\hb{n} m i$ we record first a partial fraction decomposition of the polynomial in the denominator. 
\begin{lem}[Partial fraction decomposition]
\label{LemmaPartFrac}
\[
\frac{1}{\fallfak{(r-w)}{r}}=\sum_{\ell=1}^{r}\frac{(-1)^{\ell-1}}{(\ell-1)!(r-\ell)!}\cdot\frac{1}{\ell-w}.
\]
\end{lem}
\begin{proof}
Multiplication with $\fallfak{(r-w)}{r}$ gives 
\[
1=\sum_{\ell=1}^{r}A_{r,\ell}\cdot \prod_{\substack{1\le j\le r\\ j\neq \ell}}(j-w).
\]
Evaluation at the pole $w=\ell$, $1\le \ell \le r$, gives
\[
A_{r,\ell}=\frac{1}{\prod_{\substack{1\le j\le r\\ j\neq \ell}}(j-\ell)}
=\frac{(-1)^{\ell-1}}{(\ell-1)!(r-\ell)!}
\]
\end{proof}
\begin{proof}[Proof Of Proposition~\ref{Prop1}]
The lemma stated before implies that
\[
f_i(x,y)= \sum_{\ell=1}^{i+1}\frac{(-1)^{\ell-1}}{(\ell-1)!(i+1-\ell)!}\cdot\frac{\fallfak{(i-x)}{i}}{\ell-x-y}.
\]
We restate the relation between signless Stirling numbers of the first kind:
\begin{equation*}
  \zt_{n}(\{1\}_{k}) = [q^{k}] \big(1+\frac{q}{1}\big) \cdot \big(1+\frac{q}{2}\big) \cdot \cdots \cdot \big(1+\frac{q}{n}\big).
\end{equation*}
Thus, 
\begin{equation}
\label{Expansion1}
\fallfak{(i-x)}{i}=i!\prod_{j=1}^{i}\big(1-\frac{x}{j}\big)
=i!\sum_{k\ge 0}(-1)^k\zt_{i}(\{1\}_{k})x^k.
\end{equation}
Hence, 
\begin{align*}
[x^n y^m]f_i(x,y)&=[x^n y^m]\sum_{\ell=1}^{i+1}\frac{(-1)^{\ell-1}}{(\ell-1)!(i+1-\ell)!}\frac{\fallfak{(i-x)}{i}}{\ell-x-y}\\
&=i!\sum_{\ell=1}^{i+1}\sum_{k=0}^{n}\frac{(-1)^{\ell+k-1}\zt_{i}(\{1\}_{k})}{(\ell-1)!(i+1-\ell)!}\cdot[x^{n-k}y^m]\frac{1}{\ell-x-y}\\
&=i!\sum_{\ell=1}^{i+1}\sum_{k=0}^{n}\frac{(-1)^{\ell+k-1}\zt_{i}(\{1\}_{k})}{(\ell-1)!(i+1-\ell)!}\cdot\frac1{\ell^{n-k+m+1}}\binom{n-k+m}{m}.
\end{align*}
On the other hand, the generating function of truncated zeta star values is known:
\begin{equation}
\label{GFzts}
\zts_n(\{1\}_k)=[q^{k}]\frac{1}{(1-\frac{q}{1})(1-\frac{q}{2})\dots (1-\frac{q}{n})}.
\end{equation}
This implies the following expansion:
\begin{equation}
\label{Expansion2}
\frac{1}{\fallfak{(i+1-x-y)}{i+1}}=\frac{1}{(i+1)!}\sum_{j\ge 0}\zts_{i+1}(\{1\}_j)(x+y)^j.
\end{equation}
Thus, 
\[
f_i(x,y)=\frac{\Big(\sum_{j\ge 0}\zts_{i+1}(\{1\}_j)(x+y)^j\Big)\Big(\sum_{k\ge 0}(-1)^k\zt_{i}(\{1\}_{k})x^k\Big)}{i+1}.
\]
Extraction of coefficients leads to the stated result. 
\end{proof}

\subsection{Additional properties of higher tier binomial coefficients}
In this subsection we establish various properties of the binomial coefficients of tier $i$. We discuss a generalized symmetry relation. 
Then, we obtain several different expressions for the row sums $N_i=\sum_{m+n=N}(n,m)_i$, involving the Legendre polynomials. 
Using the Euler polynomials, we show that all central tiered binomial coefficients $\hb{n}n{i}$, as given in Example~\ref{ExampleCentralTiered}, can be written in terms of those in
even tiers, i.e., with $i$ even.

\smallskip

Furthermore, results for the alternating infinite sums $\sum_{n\ge 0}(-1)^n\cdot (n,m)_i$ as well as 
$\sum_{m\ge 0}(-1)^m\cdot (n,m)_i$ are given. We also study finite sums $\sum_{i=0}^{N}(n,m)_i$ with respect to the tier $i$.
Finally, we relate the complete generating function $\sum_{n,m,i\ge 0}\hb{n} m i x^n y^m z^i$
to the Gauss hypergeometric series
\[
_2F_1(a,b,c;z)=\sum_{n\ge 0}\frac{\auffak{a}{n}\auffak{b}{n}}{\auffak{c}{n}}\cdot\frac{z^n}{n!}.
\]
Here $\auffak{a}n=a(a+1)\dots(a+n-1)$ denote the rising factorials\footnote{The rising factorials are often called Pochhammer symbols and denoted by $(a)_n=\auffak{a}n$. Due to the similarity to our binomial coefficients of tier $i$ we opted to use a notation popularized by Graham, Knuth and Patashnik~\cite{GKP}}.

\smallskip

\begin{prop}[Generalized symmetry for binomial coefficients of tier $i$]
\label{PropGenSymm}
The binomial coefficients of tier $i$ satisfy 
the generalized symmetry relation
\[
\hb{n}mi=\sum_{j=0}^{i}\binom{i}j (-1)^j \hb{m}n j.
\]
Equivalently, concerning the binomial coefficients of tier $i$, the binomial transform $\Bin$ with respect to the variable $i$, equals the operator $\Co$ with respect to $n$ and $m$:
\[
\Co ( \hb{n}mi)= \Bin (\hb{n}mi).
\]
\end{prop}
\begin{proof}
We offer two proofs. First, we note that tiered binomial
coefficients can be expressed as linear combinations of the sums $I_{n,m}^{(i)}$. Hence, 
Proposition~\ref{Prop1Symmetry} implies the stated result. 

\smallskip

On the other hand, we can concretely use the binomial theorem for the falling factorials to obtain
\[
\fallfak{(i-x)}{i}=\fallfak{(i+1-y-x +(y-1))}{i}=
\sum_{j=0}^{i}\binom{i}{j}\fallfak{(i+1-y-x)}{i-j}\fallfak{(y-1)}{j}.
\]
Consequently,
\begin{align*}
f_i(x,y)&=\frac{\fallfak{(i-x)}{i}}{\fallfak{(i+1-x-y)}{i+1}}
=\sum_{j=0}^{i}\binom{i}{j}\frac{\fallfak{(y-1)}{j}}{\fallfak{(j+1-y-x)}{j}}\\
&=\sum_{j=0}^{i}(-1)^j\binom{i}{j}\frac{\fallfak{(j-y)}{j}}{\fallfak{(j+1-y-x)}{j}}
=\sum_{j=0}^{i}(-1)^j\binom{i}{j}f_j(y,x).
\end{align*}
\end{proof}
\begin{prop}
The central tiered binomial coefficients $\hb{n}n i$
satisfy for $n\ge 0$
\[
\hb{n}n{2k+1}=-\sum_{j=0}^{k}e_{2k+1,2j}\hb{n}{n}{2j},
\] 
where $E_m(x)=\sum_{j=0}^{m}e_{m,j}x^j$ is the $m$th Euler polynomial.
\end{prop}
\begin{remark}
We note that generating function of Euler polynomials is given by 
\[
\mathcal{E}(x,t)=\sum_{m\ge 0}E_m(x)\frac{t^m}{m!}=\frac{2e^{x t}}{e^t+1}.
\]
and that $E_m(x)$ can be explicitly expressed in terms of the Bernoulli numbers $B_m$:
\[
E_m(x)=\sum_{j=0}^{m}e_{m,j}x^j=\frac{1}{m+1}\sum_{j=0}^{m+1}\binom{m+1}j(1-2^{m+1-j})B_{m+1-j}x^j.
\]
\end{remark}
\begin{proof}
We use induction with respect to $k$. From Example~\ref{TierOne} it follows that 
$\hb{n}{n}{1}=\frac12\binom{2n}n$ and thus that the conclusion holds for $k = 0$. Now suppose the
conclusion holds through $k - 1$, where $k > 0$. From Proposition~\ref{PropGenSymm} we get 
\[
2\hb{n}{n}{2k+1}=\sum_{j=0}^{2k}\binom{2k+1}j (-1)^j \hb{n}{n}j.
\]
Hence, by splitting the sum into odd and even $j$ we get
\begin{align*}
2\hb{n}{n}{2k+1}&=\sum_{j=0}^{k}\binom{2k+1}{2j} \hb{n}{n}{2j}-\sum_{\ell=0}^{k-1}\binom{2k+1}{2\ell+1}\hb{n}{n}{2\ell+1}\\
&=\sum_{j=0}^{k}\binom{2k+1}{2j} \hb{n}{n}{2j}-\sum_{\ell=0}^{k-1}\binom{2k+1}{2\ell+1}\sum_{j=0}^{i}e_{2i+1,2j}\hb{n}{n}{2j}\\
&=\sum_{j=0}^{k}\binom{2k+1}{2j} \hb{n}{n}{2j}-\sum_{j=0}^{k-1}\hb{n}{n}{2j}\sum_{\ell=j}^{k-1}\binom{2k+1}{2\ell+1}e_{2i+1,2j},
\end{align*}
so the conclusion holds if
\[
-2e_{2k+1,2j}=\binom{2k+1}{2j}+\sum_{\ell=j}^{k-1}\binom{2k+1}{2\ell+1}e_{2i+1,2j}.
\]
By a basic relation for the Euler polynomials we have to show that
\begin{equation}
\label{EulerPolynomial}
-2e_{2k+1,2j}=\sum_{p=2j}^{2k}\binom{2k+1}{p}e_{p,2j}.
\end{equation}
Now use the Euler-polynomial identities~\cite[ 23.1.6, 23.1.7]{AS}
\[
E_{2k+1}(x+1)+E_{2k+1}(x)=2x^{2k+1},\quad E_{2k+1}(x+1)=\sum_{p=0}^{2k+1}\binom{2k+1}p E_p(x)
\]
to get
\[
\sum_{p=0}^{2k+1}\binom{2k+1}p E_p(x) + E_{2k+1}(x)=2x^{2k+1}.
\]
Extract the coefficient of $x^{2j}$ , $j\le k$, to obtain
\[
\sum_{p=2j}^{2k+1}\binom{2k+1}{p}e_{p,2j}+e_{2k+1,2j}=0,
\]
from which~\eqref{EulerPolynomial} follows.
\end{proof}

\smallskip

Let $N_i$ denote the row sum $\sum_{m+n=N}(n,m)_i$.
\begin{prop}[Row sums of binomial coefficients of tier $i$]
\label{Prop2}
The row sums $N_i:=\sum_{m+n=N}\hb{n}mi$ are given by 
\[
N_i=\frac1{i+1}\sum_{\ell=0}^{N}(-1)^{\ell}2^{N-\ell}\zt_i(\{1\}_\ell)\zts_{i+1}(\{1\}_{N-\ell}).
\]
Alternatively, we have an expression in terms of the Legendre polynomials $P_n(x)=\sum_{j=0}^{n}a_{n,j}x^j$:
\[
N_i=2^{N-i}\sum_{\ell=0}^i a_{i,i-\ell}\cdot\frac{1}{(2\ell+1)^{N+1}}.
\]
Moreover, an expression with Bell polynomials is
\[
N_i=\frac{1}{N!}B_N(a_1,\dots,a_N),
\]
where $a_k=a_k(i)=(k-1)!\Big((2^k-1)H_k^{(i)}+\frac{2^k}{(i+1)^k}\Big)$.
\end{prop}

\begin{proof}
Evaluation of $f_i(x,y)$ at $x=y=z$ gives the row sum generating function 
\[
f_i(z,z)= \sum_{N\ge 0}z^N \sum_{k=0}^N\hb{N-k}{k}i=\sum_{N\ge 0}N_i\cdot z^N.
\]
Hence, $N_i=[z^N]f_i(z,z)$ and
\begin{align*}
N_i&=[z^N]\frac{\fallfak{(i-z)}{i}}{\fallfak{(i+1-2z)}{i+1}}
=\frac{1}{i+1}\sum_{\ell=0}^{N}\bigg([z^\ell]\prod_{j=1}^{i}\Big(1-\frac{z}j\Big)\bigg)
\bigg([z^{N-\ell}]\prod_{j=1}^{i+1}\Big(1-\frac{2z}j\Big)^{-1}\bigg).
\end{align*}
The expansions~\eqref{Expansion1} and~\eqref{Expansion2} give the stated result.

\smallskip

For the second representation we collect the known formula 
\[
P_n(x)=\sum_{j=0}^{n}a_{n,j}x^j=\frac1{2^n}\sum_{j=0}^{\lfloor\frac{n}2 \rfloor}
(-1)^j\binom{n}{j}\binom{2n-2j}{n}x^{n-2j}.
\]
We can rewrite the conclusion as 
\begin{equation}
\label{ConclusionLegendre}
N_i:=\sum_{m+n=N}\hb{n}mi=2^{N-2i}\sum_{j=0}^{\lfloor\frac{i}2\rfloor}(-1)^j\binom{2n-2j}{j,n-j,n-2j}(2j+1)^{-N-1}.
\end{equation}
We use the partial fraction decomposition from Lemma~\eqref{LemmaPartFrac} and obtain
\[
f_i(z,z)=\fallfak{(i-z)}{i}\cdot\sum_{\ell=1}^{i+1}\frac{(-1)^{\ell+1}}{(\ell-1)!(i+1-\ell)!}\cdot\frac{1}{\ell-2z}.
\]
We split the sum into even and odd values of $\ell$ to get
\begin{equation}
\begin{split}
\label{oddEven}
f_i(z,z)&=-\sum_{j=1}^{\lfloor\frac{i+1}2\rfloor}\frac{\fallfak{(i-z)}{i}}{(2j-1)!(i+1-2j)!}\cdot\frac{1}{2j-2z}\\
&\quad+\sum_{j=1}^{\lfloor\frac{i}2\rfloor+1}\frac{\fallfak{(i-z)}{i}}{(2j-2)!(i+2-2j)!}\cdot\frac{1}{2j-1-2z}.
\end{split}
\end{equation}
In the first sum we cancel the denominator with a factor of the numerator, leading to a polynomial of degree $i-1$.
For the second sum we derive the Laurent series around the poles $(2j-1)/2$. Symbolically, we apply the binomial theorem for the falling factorials
\[
\fallfak{(i-z)}{i}=\fallfak{\big(i-z+j-\frac12-(j-\frac12)\big)}{i}
=\sum_{k=0}^{i}\binom{i}{k}\fallfak{\big(j-\frac12-z\big)}{k}\fallfak{\big(i-j+\frac12\big)}{i-k}.
\]
The first summand gives $\fallfak{\big(i-j+\frac12\big)}{i}$. 
Thus, we get
\[
f_i(z,z)=\sum_{j=1}^{\lfloor\frac{i}2\rfloor+1}\frac{\fallfak{\big(i-j+\frac12\big)}{i}}{(2j-2)!(i+2-2j)!}\cdot\frac{1}{2j-1-2z}+p_i(z),
\]
where $p_i(z)$ is a polynomial of degree at most $i-1$, including the first sum of~\eqref{oddEven} and the remaining summands
of the expansion of $\fallfak{(i-z)}{i}$. We claim that $p_i(z)=0$. Avoiding more involved combinatorial reasoning, we argue as follows:
assume that $p_i(z)\neq 0$. Then, 
\[
\lim_{z\to\infty}f_i(z,z)=\lim_{z\to\infty}p_i(z)\in\{\pm\infty\}\cup\R\setminus\{0\}.
\]
However, by definition, the degree of the denominator of $f_i(z,z)$ is $i+1$, bigger than the degree of the numerator $i$, so
\[
\lim_{z\to\infty}f_i(z,z)=0,
\]
a contradiction. This implies that 
\[
f_i(z,z)=\sum_{j=1}^{\lfloor\frac{i}2\rfloor+1}\frac{\fallfak{\big(i-j+\frac12\big)}{i}}{(2j-2)!(i+2-2j)!}\cdot\frac{1}{2j-1-2z}.
\]
Shifting the index, simplification of $\fallfak{\big(i-j+\frac12\big)}{i}$ and extraction of coefficients then directly leads to~\eqref{ConclusionLegendre}.
\smallskip

On the other hand, we can use the $\exp-\log$ representation to get
\begin{align*}
f_i(z,z)=\frac{1}{i+1}\cdot \exp\Big(\sum_{j=1}^{i}\ln(1-\frac{z}{j})-\sum_{j=1}^{i}\ln(1-\frac{2z}{j})\Big).
\end{align*}
Series expansion of the logarithm functions give
\[
f_i(z,z)=\frac{1}{i+1}\cdot \exp\Big[\sum_{k=1}^{\infty}\frac{z^k}{k}\Big((2^k-1)H_k^{(i)}+\frac{2^k}{(i+1)^k}\Big)\Big].
\]
Thus, it follows that $N_i$ can be expressed in terms of the complete Bell polynomials $B_n(x_1,\dots,x_n)$, which are defined via
\[
\exp\Big(\sum_{\ell\ge 1}\frac{x_\ell}{\ell!}z^{\ell}\Big)
= \sum_{j\ge 0}\frac{B_j(x_1,\dots,x_j)}{j!}z^j,
\]
evaluated at $x_k=a_k(i)=(k-1)!\Big((2^k-1)H_k^{(i)}+\frac{2^k}{(i+1)^k}\Big)$.
\end{proof}

\smallskip

\begin{prop}
The infinite alternating sums $\sum_{m\ge 0}(n,m)_i\cdot (-1)^m$ satisfy
\[
\sum_{m\ge 0}\hb{n}mi\cdot (-1)^m=\frac{i+1}{(i+2)^{n+1}}-\frac{i}{(i+1)^{n+1}}.
\]
The infinite alternating sums $\sum_{n\ge 0}(n,m)_i\cdot (-1)^n$ are given by
\[
\sum_{n\ge 0}\hb{n}mi\cdot (-1)^n=\frac1{i+2}\cdot\Big(\zts_{i+2}(\{1\}_m)-\zts_{i+2}(\{1\}_{m-1})\Big).
\]
\end{prop}
\begin{proof}
The generating function of $\sum_{m\ge 0}(n,m)_i\cdot (-1)^m$ is 
\[
f_i(x,-1)=\frac{\fallfak{(i-x)}{i}}{\fallfak{(i+2-x)}{i+1}}
=\frac{1-x}{(i+2-x)(i+1-x)}=\frac{i+1}{i+2-x}-\frac{i}{i+1-x}.
\]
Extraction of coefficients gives the stated result. Similarly, the generating function of $\sum_{n\ge 0}(n,m)_i\cdot (-1)^n$ is
\[
f_i(-1,y)=\frac{(i+1)!}{\fallfak{(i+2-y)}{i+1}}
=\frac{1-y}{(i+2)\prod_{j=1}^{i+2}(1-\frac{y}{j})}.
\]
Extraction of coefficients, using the generating function of the truncated zeta star values~\eqref{GFzts}, gives the stated result. 
\end{proof}

Next we turn to the finite sums $\sum_{i=0}^{N}\hb{n}mi$ with respect to the tier $i$.
We use the following lemma, which can easily be proven using induction. 
\begin{lem}
Let $u$ and $v$ denote variables with $u\neq v+1$ and $v\notin\N$. Then, for $n\ge 0$
\[
\sum_{i=1}^{n}\frac{\binom{u}{i}}{\binom{v}i}=\frac{(v-n)\binom{u}{n+1}}{(u-v-1)\binom{v}{n+1}}-\frac{u}{u-v-1}.
\]
\end{lem}

\begin{prop}
The sums $\sum_{i=0}^{N}\hb{n}mi$ of higher tier binomial coefficients over tier $i$ are given by
\[
\sum_{j=0}^n(-1)^j\binom{n+m+1-j}{m+1}\zt_{N+1}(\{1\}_j)\zts_{N+1}(\{1\}_{n+m+1-j}).
\]
\end{prop}
\begin{proof}
The generating functions of the finite sums is given by 
\[
\sum_{i=0}^{N}f_i(x,y)=\frac{1}{1-x-y}\cdot\Big(1+\sum_{j=1}^{n}\frac{\binom{x-1}{i}}{\binom{x+y-2}{i}}\Big).
\]
Thus, we can apply the Lemma stated before and get 
\[
\sum_{i=0}^{N}f_i(x,y)=\frac{1}{1-x-y}\cdot\Big(1+\frac{x+y-N-2}{-y}\cdot\frac{\binom{x-1}{N+1}}{\binom{x+y-2}{N+1}}-\frac{x-1}{-y}\Big).
\]
Simplifications give
\[
\sum_{i=0}^{N}f_i(x,y)=\frac1y\Big(\frac{\binom{x-1}{N+1}}{\binom{x+y-1}{N+1}}-1\Big).
\]
Extraction of coefficients give
\[
\sum_{i=0}^{N}\hb{n}mi=[x^n y^{m+1}]\sum_{i=0}^{N}f_i(x,y)=\prod_{j=1}^{N+1}\frac{1-\frac{x}j}{1-\frac{x+y}j}.
\]
Proceeding as in the extraction of coefficients of $f_i(x,y)$ leads to the result.
\end{proof}

Finally, we derive the complete generating function.
\begin{prop}
The complete generating function $f(x,y,z)=\sum_{n,m,i\ge 0}\hb{n} m i x^n y^m z^i$ is given by a hypergeometric function:
\[
f(x,y,z)={}_2F_1(1-x,1,2-x-y;z).
\]
\end{prop}
\begin{proof}
We have 
\begin{align*}
f(x,y,z)
=\sum_{i\ge 0}f_i(x,y)z^i
=\frac{1}{1-x-y}\sum_{i\ge 0}\frac{\auffak{(1-x)}i\cdot \auffak{1}{i}}{\auffak{(2-x-y)}{i}}\frac{z^i}{i!},
\end{align*}
which is the stated Gauss hypergeometric series.
\end{proof}

\section{Extensions}
A natural question is if our results for the (normalized) logarithm integrals $S_{n,m}^{(i)}$ can be extended to different families of integrals. For example, one can replace the logarithms $\log^m(1-x)$ by $\log^m(1+x)$, leading to alternating infinite sums. 
On the other, the range of the parameter $i$ can be extended to negative $i$ in a certain range. In the following we discuss first the logarithmic integrals 
with negative $i$. Then, it turns out, that there is a natural generalization of $S_{n,m}^{(i)}$, closely relation to a well known special function~\cite{K1986,N1909}: 
Nielsen's polylogarithm $S_{n,m}(z)$ is defined here\footnote{In the standard convention, the number $n$ has to be replaced by $n-1$ on the right hand side of the definition. We opted to shift by one 
to be more consistent with our earlier definition} by
\begin{equation*}
S_{n,m}(z)= \frac{(-1)^{n+m}}{n!m!}\int_{0}^{1}\frac{\log^{n}(x)\log^{m}(1-zx)}{x}dx.
\end{equation*}
Setting $z=-1$ leads to alternating multiple zeta values. For $z=1$ we obtain the special instance $i=-1$ of the logarithmic integrals analyzed before:
\[
S_{n,m}(1)=S_{n,m}^{(-1)}.
\]
Let $L_{i_1,\dots,i_k}(z)$ denote the multiple polylogarithm function
\[
L_{i_1,\dots,i_k}(z)=\sum_{\ell_1>\cdots>\ell_k\ge 1}\frac{z^{\ell_1}}{\ell_1^{i_1}\cdots \ell_k^{i_k}},
\]
It is known that $S_{n,m}(z)$ is a special multiple polylogarithm function: $S_{n,m}(z)=L_{n+2,\{1\}_{m-1}}(z)$, 
such that $S_{n,1}(z)=L_{n+2}(z)$ is the ordinary polylogarithm function.

\subsection{Logarithmic integrals and negative i} 
In the following we look at the logarithmic integrals 
\[
S_{n,m}^{(-i)}=\frac{(-1)^{n+m}}{n!m!}\int_{0}^{1}\frac{\log^{n}(x)\log^{m}(1-x)}{x^i}dx,
\]
with $n\ge 0$ and $0\le i\le m$. We already know the special cases $i=0$ and $i=1$ from our investigations before.
Thus, we turn to the range $2\le i\le m$. Our first result is devoted to the special case $n=0$. 
\begin{lem}[Boundary values - case $n=0$]
\label{LemBoundaryNegativeI}
The logarithmic integrals $S_{0,m}^{(-i)}$ with $2\le i\le m$ are given by 
\begin{align*}
S_{0,m}^{(-i)}&=
\frac{1}{(i-1)!}\sum_{r=0}^{i-1}\binom{i-1}{r}\fallfak{(i-2)}{i-1-r}\\
&\quad\times \sum_{j=0}^{r}(-1)^{r-j}(r-1)!\zt_{r-1}(\{1\}_{j-1})\zt(m+1-j).
\end{align*}
\end{lem}
We use the substitution $x=1-u$ to obtain
\begin{proof}
\[
S_{0,m}^{(-i)}=\frac{(-1)^{m}}{m!}\int_{0}^{1}\frac{\log^{m}(u)}{(1-u)^i}du.
\]
Expansion of $1/(1-u)^i$ into a power series around $u=0$ gives 
\[
S_{0,m}^{(-i)}=\frac{(-1)^{m}}{m!}\int_{0}^{1}\sum_{\ell\ge 0}\binom{i-1+\ell}{\ell}u^{\ell}\log^{m}(u)du.
\]
The last integral is readily evaluated by our previous result for $S_{0,m}^{(i)}$ in Lemma~\ref{SpecialValue}.
Thus, we obtain 
\[
S_{0,m}^{(-i)}=\sum_{\ell\ge 0}\binom{i-1+\ell}{i-1}\frac{1}{(\ell+1)^{m+1}}.
\]
We use the binomial theorem for the falling factorials
\begin{align*}
\binom{i-1+\ell}{i-1}
&=\frac{\fallfak{(i-1+\ell)}{i-1}}{(i-1)!}
=\frac{\fallfak{(\ell+1+i-2)}{i-1}}{(i-1)!}\\
&=\frac{1}{(i-1)!}\sum_{r=0}^{i-1}\binom{i-1}{r}\fallfak{(\ell+1)}{r}\fallfak{(i-2)}{i-1-r}.
\end{align*}
Then, we convert the falling factorials into ordinary powers using the Stirling numbers of the first kind, or in other words
the truncated multiple zeta values of Lemma~\ref{Lem1Stir}:
\[
\fallfak{(\ell+1)}{r}
=\sum_{j=0}^{r}(-1)^{r-j}(r-1)!\zt_{r-1}(\{1\}_{j-1})(\ell+1)^j.
\]
Finally, we get the stated expression by changing summations.
\end{proof}

Next we state a recurrence relation similar to Proposition~\ref{PropHRecurrence}.
\begin{prop}
\label{PropHRecurrenceNeg}
The logarithmic integrals $S_{n,m}^{(-i)}$ with $2\le i\le m$ satisfy the recurrence relation
\[
S_{n,m}^{(-i)}=\frac{1}{i-1}\bigg(\sum_{j=2}^{i-1}S_{n,m-1}^{(-j)}-S_{n-1,m}^{(-i)}+\zt(n+1,\{1\}_{m-1})+\zt(n+2,\{1\}_{m-2})\bigg).
\]
with initial values given by $S_{0,m}^{(-i)}$ as given in Lemma~\ref{LemBoundaryNegativeI}.
\end{prop}
The proof is analogous to the result for positive $i$ and omitted. The result of Corollary~\ref{CorollNegativePowers}
follows directly using induction and the result of Borwein, Bradley and Broadhurst~\cite[Eq. (10)]{BBB}, as mentioned in the introduction.

\subsection{Nielsen's generalized polylogarithm and alternating sums.}
We generalize Nielsen's polylogarithm to include a power of $x$,
\begin{equation}
S_{n,m}^{(i)}(z)= \frac{(-1)^{n+m}}{n!m!}\int_{0}^{1}x^i\log^{n}(x)\log^{m}(1-zx)dx,
\end{equation}
with $n,m\ge 0$ and $i\ge -1$. For $i=-1$ we reobtain the ordinary Nielsen polylogarithm: $S_{n,m}^{-1}(z)=S_{n,m}(z)$. 
For $z=1$ we get the logarithmic integrals: $S_{n,m}^{(i)}(1)=S_{n,m}^{(i)}$. The special case $z=0$ leads to $S_{n,m}^{(i)}(0)=S_{n,0}^{(i)}$, which was collected before in Lemma~\ref{SpecialValue}.
Thus, we assume in the following that $z\neq 0$. 

\smallskip

It will turn out that a structurally similar recurrence relation to the logarithmic integrals $S_{n,m}^{(i)}$ also holds for $S_{n,m}^{(i)}(z)$.
Interestingly, it turns out that the case $z=1$ treated before is very special. Only for $z=1$ the generalized symmetry relations in Propositions~\ref{Prop1Symmetry} and~\ref{PropGenSymm} hold.

\smallskip

First, we turn to the boundary values. Obviously, we have $S_{n,0}^{(i)}(z)=S_{n,0}^{(i)}$. We turn to the remaining case of $n=0$. 
\begin{lem}[Boundary values - case $n=0$]
\label{LemBoundaryNielsen}
For $z\neq 0$ the generalized Nielsen polylogarithms $S_{0,m}^{(i)}(z)$ with $i\ge 0$ are given by
\begin{align*}
S_{0,m}^{(i)}(z)&=\frac{(-1)^{m}}{m! z^i}\sum_{j=0}^{i}\binom{i}{j}(-1)^{j}\bigg[(1-z)^{j+1}\sum_{\ell=0}^{m-1}\frac{(-1)^{\ell+1}\cdot \fallfak{m}{\ell}}{z(j+1)^{\ell+1}}\log^{m-\ell}(1-z)\\
&\qquad+(-1)^{m+1}m!\frac{(1-z)^{j+1}-1}{z(j+1)^{m+1}}\bigg].
\end{align*}
\end{lem}
\begin{proof}
We write $x=-(1-zx-1)/z$ and use the binomial theorem to get
\begin{align*}
S_{0,m}^{(i)}(z)
= \frac{(-1)^{m}}{m! z^i}\sum_{j=0}^{i}\binom{i}{j}(-1)^{j}\int_{0}^{1}(1-zx)^j\log^{m}(1-zx)dx.
\end{align*}
We evaluate the remaining integrals by establishing a recurrence relation:
\begin{align*}
\int_{0}^{1}(1-zx)^j\log^{m}(1-zx)dx&=
-\frac{(1-z)^{j+1}}{(j+1)z}\log^{m}(1-z)dx\\
&\qquad-\frac{m}{j+1}\int_0^1(1-zx)^{j}\log^{m-1}(1-zx)dx,
\end{align*}
$m\ge 1$. This implies that
\begin{align*}
\int_{0}^{1}(1-zx)^j\log^{m}(1-zx)dx&
=(1-z)^{j+1}\sum_{\ell=0}^{m-1}\frac{(-1)^{\ell+1}\cdot \fallfak{m}{\ell}}{z(j+1)^{\ell+1}}\log^{m-\ell}(1-z)\\
&\qquad+(-1)^{m+1}m!\frac{(1-z)^{j+1}-1}{z(j+1)^{m+1}}.
\end{align*}

\end{proof}
Next we state an expression similar to Proposition~\ref{PropStir}.
\begin{prop}
\label{PropExpressionNielsen}
The generalized Nielsen's polylogarithms $S_{n,m}^{(i)}(z)$ satisfies
\[
S_{n,m}^{(i)}(z)=\sum_{\ell=1}^{\infty}\frac{z^{\ell}}{\ell(\ell+i+1)^{n+1}}\cdot \zt_{\ell-1}(\{1\}_{m-1}.
\]
\end{prop}
The proof is identical to the proof of Proposition~\ref{PropStir} and left to the reader. 

\smallskip 

Finally, we turn to a recurrence relation in the style of Propositions~\ref{PropHRecurrence} and~\ref{PropHRecurrenceNeg}.
\begin{prop}
For non-zero $z$ and $n,m\ge 1$ the generalized Nielsen's polylogarithms $S_{n,m}^{(i)}(z)$ satisfy the recurrence relation
\[
S_{n,m}^{(i)}(z)=\frac{1}{i+1}\cdot
\bigg[
S_{n-1,m}^{(i)}(z) +\sum_{j=0}^{i}\frac{1}{z^{i-j}}S_{n,m-1}^{(j)}(z)-\frac{1}{z^{i+1}}L_{n+1,\{1\}_{m-1}}(z)
\bigg],
\]
with initial values $S_{n,0}^{(i)}(z)=S_{n,0}^{(i)}$, as given in Lemma~\ref{SpecialValue} 
and initial values $S_{0,m}^{(i)}(z)$ as stated in Proposition~\ref{LemBoundaryNielsen}.
\end{prop}
\begin{remark}
We note that the boundary values $S_{0,m}^{(i)}(z)$ also satisfy for $m\ge 1$ a recurrence relation of this form:
\[
S_{0,m}^{(i)}(z)=\frac{1}{i+1}\cdot
\bigg[
L_{\{1\}_{m}}(z) +\sum_{j=0}^{i}\frac{1}{z^{i-j}}S_{0,m-1}^{(j)}(z)-\frac{1}{z^{i+1}}L_{\{1\}_{m}}(z)
\bigg].
\]
\end{remark}
\begin{proof}
We follow closely the proofs of Lemma~\ref{LemRelationST} and Proposition~\ref{PropHRecurrence}.
Our starting point is Proposition~\ref{PropExpressionNielsen} stated before. We obtain the recurrence relation
in two different ways: first, by partial fraction decomposition and second, by the recurrence relation 
for $\zt_{\ell-1}(\{1\}_{m-1}$. Introducing generalized $T$-values~\eqref{Hurwitz} with $i\ge -1$ and $n,m\ge0$,
\[
T_{n,m}^{(i)}(z) = \sum_{\ell=1}^{\infty}\frac{z^{\ell}}{(\ell+i+1)^{n+1}}\zt_{\ell-1}(\{1\}_{m-1}),
\]
with initial values 
\[
T_{n,m}^{(-1)}(z) 
L_{n+1,\{1\}_{m-1}}(z),
\]
we observe that by partial fraction decomposition
\[
S_{n,m}^{(i)}(z)=\frac{1}{i+1}\Big[S_{n-1,m}^{(i)}-T_{n,m}^{(i)}(z)\Big],
\]
$n\ge 0$ and $m\ge 1$. 
By the recurrence relation for the truncated zeta values we get
\[
S_{n,m}^{(i)}(z)=\frac1z T_{n,m+1}^{(i-1)}(z)-T_{n,m+1}^{(i)}(z),\quad i\ge 0.
\]
Thus, 
\[
T_{n,m+1}^{(i)}(z)=-\sum_{j=0}^{i}\frac{1}{z^{i-j}}S_{n,m}^{(i)}(z)+\frac{1}{z^{i+1}}T_{n,m+1}^{(-1)}(z).
\]
This leads to the stated result.

\end{proof}

\section{Applications - Moments of the quicksort limit law}
\subsection{Quicksort algorithm} 
Quicksort is a famous sorting algorithm invented by Hoare~\cite{Hoare1962}. It has been analyzed in a great many papers under the so-called uniform random model. This means that the input is a random permutation of size $n$. One of the most popular cost measures is the number of comparison $C_n$ required to sort a list of length $n$. 
Under the uniform random model the number $C_n$ becomes a random variable, satisfying the stochastic recurrence relation
\[
C_n\law C_{I_n} + C^{\ast}_{n-I_n}+n-1,
\]
$n\ge 2$ with initial values $C_0=0$ and $C_1=0$. Here, the random variables $C^{\ast}_n$ denote independent copies of the $C_n$, and the random variable $I_n$ is a discrete uniformly distributed random variable on the set $[n]=\{1,\dots,n\}$, independent of the $C$ and $C^{\ast}$. The expected value and the variance are readily obtained by taking expectations, and go back to Knuth~\cite{Knu1973}:
\[
\E(C_n)=2(n+1)H_n-4n,
\]
\[
\V(C_n)\sim (7-4\zeta(2))n^2=(7-\frac23\pi^2)n^2.
\]
See Zeilberger~\cite{Zeil2019} for a modern computer algebra approach on exact and asymptotic expressions for the moments.

\smallskip

R\'egnier~\cite{Reg1989} used martingale theory to prove that the normalized and centered random variable
\[
Z_n:=\frac{C_n-\E(C_n)}{n+1}
\]
converges to a non-degenerate limit $Z_\infty$, both almost surely and in $L_p$ for all $1\le p <\infty$. 
Hennequin~\cite{He89,He91} calculated all the moments of $Z_\infty$ and characterized its cumulants, also for a great many variants of the quicksort algorithm. R\"osler~\cite{Roe1991} constructed a random variable satisfying a distributional equation
\begin{equation}
\label{QSdist}
Z\law U\cdot Z_1 + (1-U)\cdot Z_2+C(U),
\end{equation}
where the random variables $Z_1,Z_2$ are independent copies of $Z$, $U$ is a standard uniformly distributed random variable, 
independent of the $Z$ variables, and the toll function $C(x)$ is given by the entropy function
\[
C(x)=1+2x\ln(x)+2(1-x)\ln(1-x).
\]
Here the random variable $Z$ has the same distribution as $Z_\infty$. The moments of $C_n$, as well as many related stochastic recurrence relations, have been studied by Hwang and Neininger~\cite{HN2002}. For the sake of completeness, we mention that the limit law of $C_n$ has been refined by Neininger~\cite{N2015} (see also Fuchs~\cite{F2015}, Gr\"ubel and Kabluchko~\cite{GK2016} and Sulzbach~\cite{Sulzbach}):
\[
\sqrt{\frac{n}{2\ln n}}\big(Z_n-Z_\infty\big)\to\mathcal{N}(0,1).
\] 

\smallskip

Our goal is to obtain additional structural information about the limit law $Z=Z_\infty$. We complement Hennequin's recurrence relation for the cumulants~\cite{He91} by adding a new one involving the binomial coefficients $\hb{n}mi$ of tier $i$.

\smallskip

By taking the $s$th power, $s\ge 1$, we have 
\[
Z^s\law \sum_{k_1+k_2+k_3=s}\binom{s}{k_1,k_2,k_3}U^{k_1}(1-U)^{k_2}C^{k_3}(U)\cdot Z_1^{k_1}\cdot Z_2^{k_2}.
\]
Let $\mu_s=\E(Z^s)$. We have $\mu_0=1$, $\mu_1=0$ and the recurrence relation
\begin{equation*}
\mu_s= \sum_{k_1+k_2+k_3=s}\binom{s}{k_1,k_2,k_3}\mu_{k_1}\mu_{k_2}\int_0^{1}x^{k_1}(1-x)^{k_2}C^{k_3}(x)dx,
\end{equation*}
such that 
\begin{equation}
\label{QSrec}
\mu_s= \frac{s+1}{s-1}\cdot\sum_{\substack{k_1+k_2+k_3=s\\ k_1,k_2<s}}\binom{s}{k_1,k_2,k_3}\mu_{k_1}\mu_{k_2}\int_0^{1}x^{k_1}(1-x)^{k_2}C^{k_3}(x)dx,\quad s>1.
\end{equation}

\smallskip

By our previous results on $I_{n,m}^{(i)}$ we reobtain the following result of Hennequin~\cite{He91} directly from the distributional equation of $Z$.
\begin{thm}
The moments of the quicksort limit law $Z_\infty$ are rational polynomials in the ordinary zeta values.
\end{thm}
\begin{proof}
We use induction with respect to $s$. The statement is true for $s=1,2$: $\mu_1=0$ and $\mu_2=7-4\zt(2)$. 
Assuming that the statement holds for $1\le k\le s-1$ the recurrence relation~\eqref{QSrec} 
leads to the result, since $\int_0^{1}x^{k_1}(1-x)^{k_2}C^{k_3}(x)dx$ can be decomposed a sum of logarithmic integrals 
$I_{n,m}^{(i)}$, which are by Corollary~\ref{CorollRatioPoly} always rational polynomials in the zeta values.
\end{proof}
Let $c_s=\CT(\mu_s)$ denote the constant term in the expression of $\mu_s$ in terms of rational polynomials in zeta values. 
We obtain the following recurrence relation for the sequence $(c_s)_{s\in\N}$. 
\begin{prop}[Constant term - Quicksort limit law]
\label{QSCT}
The sequence $c_s$ satisfies $c_0=1$, $c_1=0$  and for $s\ge 1$
\begin{align*}
c_s&= \frac{s+1}{s-1}\cdot\sum_{\substack{k_1+k_2+k_3=s\\ k_1,k_2<s}}\binom{s}{k_1,k_2,k_3}c_{k_1}c_{k_2}\sum_{n+m+p=k_3}\\
&\quad\times \sum_{j=0}^{m+k_2}\binom{k_3}{n,m,p}\binom{m+k_2}{j}\cdot (-1)^j2^{n+m}(-1)^{n+m}n!m!\hb{n}m{n+k_1+j},
\end{align*}
with $\hb{n}mi$ denote the binomial coefficients of tier $i$.
\end{prop}

\begin{remark}[Normalized constant terms]
The recurrence relation for $c_s$ suggests to look at the normalized sequence $\tilde{c}_s$, defined by
\[
\tilde{c}_s=(-1)^s\frac{c_s}{s!2^s},
\] 
with initial values $\tilde{c}_0=c_0=1$ and $\tilde{c}_1=c_1=0$. 
The recurrence relation takes a particularly simple form:
\begin{align*}
\tilde{c}_s&= \frac{s+1}{s-1}\cdot\sum_{\substack{k_1+k_2+k_3=s\\ k_1,k_2<s}}\tilde{c}_{k_1}\tilde{c}_{k_2}\sum_{n+m+p=k_3}\frac{(-1)^p}{p!2^p}\\
&\quad\times \sum_{j=0}^{m+k_2}\binom{m+k_2}{j}\cdot (-1)^j\hb{n}m{n+k_1+j}.
\end{align*}
\end{remark}

\begin{proof}
By~\eqref{SUMbor} we have
$\CT(\zt(a+1,\{1\}_{b})=0$. Consequently, Theorem~\ref{ThmMain} implies that
$\CT(S_{n,m}^{(i)})=\hb{n}mi$. We expand $C^{k_3}(x)$ using the multinomial theorem, 
and then expand again $(1-x)^{m+k_2}$ using the binomial theorem.
Application of the constant term operator to~\eqref{QSrec} gives the stated result.
\end{proof}

\smallskip 

\subsection{Cumulants}
The cumulants of a random variable $Z$ are given by the expansion of the logarithm of the moment generating function
\[
M(t)=\E(e^{Z t})=\sum_{s\ge 0}\frac{\mu_s}{s!}t^s,\quad K(t)=\log(M(t))=\sum_{s\ge 1}\frac{\kappa_s }{s!}t^s.
\]
The $s$th cumulant is homogeneous of degree $s$; furthermore cumulants of order greater than one are shift invariant.
Due to the relation
\[
M(t)=\exp(K(t)),
\]
the cumulants are related to the ordinary moments by the complete Bell polynomials
\[
\mu_s=B_s(\kappa_1,\dots,\kappa_s).
\]
Likewise, the cumulants are given in terms of the moments as
\[
\kappa_s=\sum_{j=1}^{s}(-1)^{j-1}(j-1)!B_{s,j}(\mu_1,\dots,\mu_{s-j+1}).
\]
Here, the $B_{s,j}(\mu_1,\dots,\mu_{s-j+1})$ denote the partial or incomplete Bell polynomials.

\begin{example}[Gumbel distribution - cumulants]
The Gumbel distribution is an extreme value distribution with density function
\[
f(x)=e^{-x-e^{-x}},\quad x\in\R. 
\]
Its expected value $\E(X)=\gamma\doteq 0.5772\dots$ is given by the Euler-Mascheroni constant.
The cumulants have a particularly appealing form: $\kappa_1(X)=\E(X)=\gamma$, $\kappa_s(X)=(s-1)!\zt(s)$ for $s>1$.
Consequently, the centered and scaled Gumbel-distributed random variable $G=-2(X-\gamma)$ satisfies
\[
\kappa_1(G)=0,\quad \kappa_s(G)=(-1)^s\cdot 2^s\cdot (s-1)!\zt(s),\ s>1.
\]
\end{example}

Hennequin~\cite{He91} obtained the cumulants of the quicksort limit law.
\begin{thm}[Cumulants of Quicksort limit]
The cumulants $\kappa_s=\kappa_s(Z)$ of the limit law $Z$ of the Quicksort comparisons satisfy $\kappa_1=0$ and
\[
\kappa_s=a_s + (-1)^{s+1}\cdot 2^s \cdot (s-1)!\zt(s),\quad s\ge 2.
\]
Here the $a_s$ are determined by a certain recurrence relation~\cite{He91}.
\end{thm}
Let $Z^{\ast}=Z+G$ denote the shifted limit law of the Quicksort comparisons, with
$G=-2(X-\gamma)$, $X$ a Gumbel-distributed random variable, independent of $Z$. 
By the properties of the cumulants of independent random variables 
\[
\kappa_s(Z^{\ast})=\kappa_s(Z+G)=\kappa_s(Z)+\kappa_S(G)=a_s\in\Q
\]
with $\kappa_1(Z^{\ast})=0$, $\kappa_2(Z^{\ast})=7$ and so on.

\smallskip

Equivalently, in terms of moments, let $\mu^{\ast}_s=\E\big((Z^{\ast})^s\big)$. Then,
\[
\mu^{\ast}_s=\sum_{j=0}^{s}\binom{s}j \mu_{j,Z} \cdot\mu_{s-j,G}\in\Q.
\]
with $\mu_j=\mu_{j,Z}$ given by~\eqref{QSrec} and 
\[
\mu_{G,s}=B_s(0,2^2\cdot 1!\zt(2),\dots,(-1)^s2^s\cdot(s-1)!\zt(s)).
\]
We note first that the moments $\mu_{G,s}$ are rational polynomials in the zetas with constant term given 
by the Kronecker delta: $\CT(\mu_{G,s})=\delta_{0,s}$.
Hence, the rational number 
$\mu^{\ast}_s$ is given by
\[
\mu^{\ast}_s=\CT(\sum_{j=0}^{s}\binom{s}j \mu_{j,Z} \cdot\mu_{s-j,G})
=\CT(\mu_{s,Z})=c_s.
\]

\begin{prop}[Cumulants of the shifted Quicksort limit]
Let $Z^{\ast}=Z+G$ denote the shifted limit law of the Quicksort comparisons, with
$G=-2(X-\gamma)$, $X$ a Gumbel-distributed random variable, independent of $Z$. 
Then, $\kappa_s(Z^{\ast})\in\Q$, with 
\[
a_s=\kappa_s(Z^{\ast})=\sum_{j=1}^{s}(-1)^{j-1}(j-1)!B_{s,j}(c_1,\dots,c_{s-j+1})
\] 
given in terms of the constant terms $c_s$ in Proposition~\ref{QSCT}, involving the binomial coefficients of tier $i$.
\end{prop}

\section{Conclusion}
We studied the logarithmic integral
\[
I_{n,m}^{(i)}  :=\int_0^1 x^i\cdot \ln^n(x)\cdot \ln^m(1-x)\ dx
\]
and the normalized values $S_{n,m}^{(i)}$ for $i\ge -1$. We determined a recurrence relation
\[
S_{n,m}^{(i)}=\frac{1}{i+1}\cdot
\bigg[
S_{n-1,m}^{(i)} +\sum_{j=0}^{i}S_{n,m-1}^{(j)}-\zt(n+1,\{1\}_{m-1})
\bigg],
\]
leading to an expansion into elements $\zt(n+1,\{1\}_{m-1})$:
\[
S_{n,m}^{(i)}=\hb{n}mi - \sum_{\substack{1\le a\le n\\a\le b\le m}}
\hb{n-a}{m-b}i\cdot\zt(a+1,\{1\}_b).
\]
Additionally, this gives an evaluation of $S_{n,m}^{(i)}$ into ordinary zeta values. The binomial coefficients $\hb{n}mi$ of tier $i$ are rational numbers 
and appear in the expansion of $I_{n,m}^{(i)}$. Various expressions and properties of $\hb{n}mi$ are established. 
We also considered variants and extensions of the logarithmic integrals, providing recurrence relations, which may serve as a starting point for expansions
similar to the one provided here in our main theorem. As an application of our results, we revisited the moments and cumulants of the number of comparisons in quicksort, 
relating them to binomial coefficients $\hb{n}mi$ of tier $i$.

\end{document}